\documentclass[12pt]{amsart}

\usepackage{mathtools}
\usepackage{titlesec,caption, subcaption}
\usepackage{tikz-cd}
\usepackage{amsmath,amsthm,amsfonts,amscd,amssymb,amsopn,enumerate,enumitem}
\usepackage{mathrsfs,xcolor,hyperref,stmaryrd}
\usepackage{wasysym}
\usepackage[all]{xy}
\usepackage{fullpage}
\usepackage{verbatim}
\usepackage{colonequals}
\usepackage{graphicx}
\usepackage[OT2,OT1]{fontenc}
\usepackage[new]{old-arrows}

\titleformat{name=\section}{}{\thetitle.}{0.8em}{\centering\scshape}
\titleformat{name=\subsection}[runin]{}{\thetitle.}{0.8em}{\bfseries}[.]
\titleformat{name=\subsubsection}[runin]{}{\thetitle.}{0.8em}{\bfseries}[.]

\theoremstyle{plain}
\newtheorem{thm}{Theorem}[section]
\newtheorem*{thm*}{Theorem}

\newtheorem*{conj*}{Conjecture}

\newtheorem*{fact*}{Fact}
\newtheorem*{prop*}{Proposition}

\theoremstyle{definition}

\newtheorem*{defn*}{Definition}

\newtheorem*{assu*}{Assumptions}
\newtheorem{rem}[thm]{Remark}

\newtheorem{prop}[thm]{Proposition}
\newtheorem{theorem}[thm]{Theorem}
\newtheorem{lemma}[thm]{Lemma}
\newtheorem{cor}[thm]{Corollary}
\newtheorem*{ack*}{Acknowledgements}
\theoremstyle{remark}

\title{On the modularity of elliptic curves over the cyclotomic $\mathbb{Z}_p$-extension of some real quadratic fields}
\author{Xinyao Zhang}

\begin{document}
	
	\maketitle
	\begin{abstract}	
		The modularity of elliptic curves always intrigues number theorists. Recently, Thorne had proved a marvelous result that for a prime $ p $, every elliptic curve defined over a $ p $-cyclotomic extension of $ \mathbb{Q} $ is modular. The method is to use some automorphy lifting theorems and study non-cusp points on some specific elliptic curves by Iwasawa theory for elliptic curves. Since the modularity of elliptic curves over real quadratic was proved, one may ask whether it is possible to replace $ \mathbb{Q} $ with a real quadratic field $ K $. Following Thorne's idea, we give some assumptions first and prove the modularity of elliptic curves over the $\mathbb{Z}_p$-extension of some real quadratic fields.
	\end{abstract}

\footnotetext{\textit{Key words and phrases. Modularity, Iwasawa theory for elliptic curves.}}

	\section{Introduction}
	
	Let $E_1$, $E_2$ be two elliptic curves given by the following Weierstrass equations.
	
	\begin{equation*}
		E_1: y^2+xy+y=x^3+x^2-10x-10.
	\end{equation*}
    \begin{equation*}
	E_2: y^2+xy+y=x^3+x^2-5x+2.
    \end{equation*}
    In order to state our main theorem, we give the following assumptions first.
    \begin{assu*} Let $ i=1,2 $. Let $K$ be a real quadratic field, $p$ an odd prime satisfying \\	
    	1. $ K \ne \mathbb{Q}(\sqrt{5})$,\\
    	2. The Mordell-Weil group $E_i(K)$ is finite,\\
    	3. The $p$-Tate-Shafarevich group {\fontencoding{OT2}\selectfont X}$(E_i/K)_p$ is trivial, \\
    	4. $p$ splits completely  in $K$.
    	
    \end{assu*}
    Our goal in this article is to prove the following result:
    \begin{theorem}
    	Let $ K $ and $ p $ satisfy the \textbf{Assumptions}. Let $F$ be a number field contained in the cyclotomic $\mathbb{Z}_p$-extension of $K$. Let $E$ be an elliptic curve over $ F $. Then $ E $ is modular.
    \end{theorem}
    
    In recent years, a number of results about modularity or automorphy lifting theorems have been proved. As arithmetic consequences, the modularity of elliptic curves over some specific fields can be verified. 
    By studying non-cuspidal points on some modular curves, the following theorems had been proved:
    \begin{theorem}	[\cite{Freitas2015}, Theorem 1]
    Let $ E $ be an elliptic curve over a real quadratic field $ K $. Then $ E $
    is modular.	
     \end{theorem}
    \begin{theorem}[\cite{Derickx2020}, Theorem 4]
    	Let $ K $ be a totally real cubic number field. Let $ E $ be an elliptic curve  over $  K $. Then $ E $ is modular.
    \end{theorem}
     For more discussions, one can see \cite{Hung2013}.
     
    Using Iwasawa theory for elliptic curves, Thorne proved the following remarkable result:
    \begin{theorem}[\cite{Thorne2019}, Theorem 1]
    	Let $ p $ be a prime, and let $ F $ be a number field which is contained in the cyclotomic $\mathbb{Z}_p$-extension
    	of $ \mathbb{Q} $. Let $ E $ be an elliptic curve over $ F $. Then $ E $ is modular.
    \end{theorem}

	To prove Theorem 1.1, our strategy is to extend Thorne's method, but there are some obstrucions to study the $ K_n $-points of $ E_i $, where $ K_n $ is the $ p $-cyclotomic extension of a real quadratic field $ K $ satisfying $ [K_n:K]=p^n $. In Thorne's proof, a key point is the fact that the Selmer group  $ \operatorname{Sel}(E_i/ \mathbb{Q} )_p $ is trivial and then one can use some results in Iwasawa theory. This is the reason why we first assume that the Mordell-Weil group $E_i(K)$ is finite and the $p$-Tate-Shafarevich group {\fontencoding{OT2}\selectfont X}$(E_i/K)_p$ is trivial. Under these assumptions, we use some more general results in Iwasawa theory for elliptic curves to prove our result.

	\section{Some data}
	
	In this section, we will give some data of $ E_i $, most of which are from \cite{LMFDBCollaboration2022}.\\
	
	\begin{center}	
	 
	\begin{tabular}{|rrr|}

		\hline
		Information & $ E_1 $ & $E_2$\\
		\hline
		Mordell-Weil group & $\mathbb{Z}/{2}\mathbb{Z} \times \mathbb{Z}/{4}\mathbb{Z}$ & $\mathbb{Z}/{2}\mathbb{Z} \times \mathbb{Z}/{4}\mathbb{Z}$\\
		Conductor & $ 15 $ & $ 15 $\\
		$ j $-invariant &$\frac{111284641}{50625} =3^{-4} \cdot 5^{-4} \cdot 13^{3} \cdot 37^{3}$ &$\frac{13997521}{225}= 3^{-2} \cdot 5^{-2} \cdot 241^{3} $\\
		Reduction type & semistable & semistable\\
		~& split multiplicative: $ 5 $ & split multiplicative: $ 5 $ \\
		~& non-split multiplicative: $ 3 $ & non-split multiplicative: $ 3 $ \\
		Tamagawa product & $ 8 $ & $ 4 $ \\
		Analytic order of {\fontencoding{OT2}\selectfont X} & $ 1 $ &$ 1 $\\
		Growth of $ 2^{\infty} $-torsion & $\mathbb{Q}(\sqrt{5}): \mathbb{Z}/{2}\mathbb{Z} \times \mathbb{Z}/{8}\mathbb{Z}$& $\mathbb{Q}(\sqrt{5}): \mathbb{Z}/{2}\mathbb{Z} \times \mathbb{Z}/{8}\mathbb{Z}$ \\
		~ & $\mathbb{Q}(i): \mathbb{Z}/{4}\mathbb{Z} \times \mathbb{Z}/{4}\mathbb{Z}$ & $\mathbb{Q}(i,\sqrt{15}): \mathbb{Z}/{4}\mathbb{Z} \times \mathbb{Z}/{4}\mathbb{Z}$\\
		\hline
		\end{tabular}
		
\end{center}
\begin{center}
	\textbf{Table 1}. Some data of $ E_i $.
\end{center}
~\\

   From table 1, we can obtain:
   \begin{lemma}
   	The Tamagawa product $ \operatorname{Tam} (E_i/K) $ is a power of $ 2 $.
   \end{lemma}
\begin{proof}
	If $ E_i $ has nonsplit multiplicative reduction at some $ v $, the Tamagawa factor at $ v $ is either $ 1 $ or $ 2 $ (see \cite{Silverman2009}, Appendix C, proof of Corollary 15.2.1). If $ E_i $ has split multiplicative reduction at some $ v $, by Kodaira and N\'eron's theorem (\cite{Silverman2009}, Theorem 6.1), the Tamagawa factor at $ v $ is a power of $ 2 $ from the decomposition of the $ j $-invariant. The lemma holds since $ E_i $ is semistable.
	
\end{proof}
\begin{lemma}
	Under our assumptions, $ E_i(K)[2^{\infty}] =E_i(K_\infty)[2^{\infty}]$, where $ K_\infty $ is the cyclotomic $\mathbb{Z}_p$-extension of $K$.
\end{lemma}
\begin{proof}
	Let $ \rho_{E_i,2} $ be the Galois representation associated to the Tate module $ T_2E_i $. As the elements in $ E_i[2] $ are $ \mathbb{Q} $-rational, the image of the absolute Galois group $ G_\mathbb{Q} $ is congruent to the identity element$\mod 2\mathbb{Z}_2$ and thus a pro-$ 2 $ group. As $ K $ and $ K_\infty $ are both Galois extension of $ \mathbb{Q} $, the index of $ G_K $ and $ G_{K_{\infty}} $ are both a power of $ 2 $. Then the lemma holds since $ p $ is an odd prime.
\end{proof}
\begin{rem}
In fact, under our assumptions, we have $ E_i(\mathbb{Q})[2^{\infty}] =E_i(K_\infty)[2^{\infty}]$ from the growth of $ 2^{\infty} $-torsion of $ E_i $ in number fields given above.
\end{rem}

\section{The proof}

   For a number field $ F $, we denote the absolute Galois group by $ G_F $. For an elliptic curve $ E $ over $ F $ and a rational prime $ l $, let $ \bar{\rho}_{E,l} : G_F \to \operatorname{GL_2}(\mathbb{F}_l) $ be the Galois representation associated to the $ l $-torsion points of $ E $. Let $ \zeta_{l} $ be a primitive $ l $-th root of unity.
	
   \begin{theorem}[\cite{Thorne2019}, Theorem 2]
   	Let $ E $ be an elliptic curve over a totally real number field $ F $, and suppose that (at least) one of the following is true:\\
   	1. The representation $ \bar{\rho}_{E,3}|_{G_{F(\zeta_{3})}}$ is absolutely irreducible.\\
   	2. $ \sqrt{5} \notin F$, and $ \bar{\rho}_{E,5}$ is irreducible.\\
   	Then $ E $ is modular.
   \end{theorem}	

    Following \cite{Freitas2015} and \cite{Thorne2019}, we consider modular curves $ X(s3,b5) $ and $ X(b3,b5) $. The next corollary is an analogue to [\cite{Thorne2019}, Lemma 3] .
    
    \begin{cor}
    	Let $ F $ be a totally real field such that $ \sqrt{5} \notin F$.\\
    	1. If $ E/F $ is an elliptic curve which is not modular, then $ E $ determines an $ F $-rational point of one of the curves $ X(s3, b5) $, $ X(b3, b5) $.\\
    	2. Recall that $K$ is a real quadratic field satisfying $ K \ne \mathbb{Q}(\sqrt{5})$. If $ F/K $ is cyclic and $ X(s3, b5)(F) = X(s3, b5)(K) $, $ X(b3, b5)(F) = X(b3, b5)(K) $, then all elliptic curves
    	over $ F $ are modular.
    \end{cor}

    \begin{proof}
    	The first assertion is a consequence of Theorem 3.1 and [\cite{Freitas2015}, Proposition 4.1]. The second part follows from the first one, Theorem 1.2 and base change for $  \operatorname{GL_2} $ (see \cite{Langlands1980}).
    \end{proof}

    \begin{prop}[\cite{Thorne2019}, Proposition 4] 
    	The modular curves $ X(b3, b5) $, $ X(s3, b5) $ are isomorphic over $ \mathbb{Q} $ to $ E_1 $, $ E_2 $, respectively.
    \end{prop}
	\begin{rem}
		$ E_1 $ and $ E_2 $ are elliptic curves of Cremona label $ 15A1$ and $ 15A3 $, respectively. See \cite{Cremona1997}.
	\end{rem}

    \begin{lemma}
    	Let $ l $ be an odd prime. Under our assumptions, we have $  E_i(\mathbb{Q})[l^{\infty}] =E_i(K_\infty)[l^{\infty}] $.
    \end{lemma}
	\begin{proof}
		Considering the Galois representation $ \bar{\rho}_{E,l} : G_{\mathbb{Q}} \to \operatorname{GL_2}(\mathbb{F}_l) $, we know that $ \bar{\rho}_{E,l}$ is surjective (see [\cite{Thorne2019}, proof of Theorem 5] or [\cite{Yoshikawa2016}, Lemma 2.3]). If $ K_{\infty} $ has an $ l $-torsion point, then the order of $ \bar{\rho}_{E,l}|_{G_{K_{\infty}}}$ is a factor of $ l(l-1) $. Since $ [K:\mathbb{Q}] =2$ and $ K_{\infty} $ is the $\mathbb{Z}_p$-extension of $ K $, the index of $ \bar{\rho}_{E,l}|_{G_{K_{\infty}}}$ in $ \operatorname{GL_2}(\mathbb{F}_l) $ is either $ p^n $ or $ 2p^n $ for some $ n $. It is known that the order of the finite group $ \operatorname{GL_2}(\mathbb{F}_l) $ is $ l(l+1)(l-1)^2 $, so we have $ (l+1)(l-1) \mid 2p^n $. The only possibility is $ p=2 $, but it is against our assumptions.
	\end{proof}
     \begin{cor}
     	Under our assumptions, the $ p $-Selmer group $ \operatorname{Sel}(E_i/K)_p $ is trivial.
     \end{cor}
     \begin{proof}
     	It is a direct consequence of Lemma 3.5.
     \end{proof}
      
      To prove the next proposition, we will use Iwasawa theory. We denote a generator of the characteristic ideal of the Pontryagin dual of the Selmer group $ \operatorname{Sel}(E_i/K_{\infty})_p $ by $ f_{E_i} $. Let the $ \lambda $-invariant $ \lambda_{E_i} $ be the $\mathbb{Z}_p $-corank of $ \operatorname{Sel}(E_i/K_{\infty})_p $.
      
      \begin{prop}
      	Under our assumptions, we have $E_i(K) =E_i(K_\infty)$.
      \end{prop}
     \begin{proof}
     	By Lemma 2.2 and Lemma 3.5, we just need to prove that $ E_i(K_\infty) $ is finite. Since $ p $ splits completely in $ K $, we can divide the proof into three cases. \\
     	
        Case 1. $ E_i $ has good ordinary reduction at all primes lying above $ p $.\\
        
        In this case, we have the formula [\cite{Greenberg1999}, Theorem 4.1]:
        \begin{equation*}
        	f_{E_i}(0) \sim\left(\prod_{v~\mathrm{bad}} c_{v}^{(p)}\right)\left(\prod_{v \mid p}\left|\tilde{E_{i}}_{,v}\left(\mathbb{F}_{v}\right)_{p}\right|^{2}\right)\left|\operatorname{Sel}(E_i/K)_{p}\right| /\left|E_i(K)_{p}\right|^{2}.
        \end{equation*}
     	Here $ a\sim b $ when $ a $ and $ b $ have the same $ p $-adic valuation, the highest power of $ p $ dividing  the Tamagawa factor $ c_v $ at $ v $ is $ c_{v}^{(p)} $, and $ \mathbb{F}_{v} $ is the residue field at $ v $. 
     	
     	As $ |E_i(\mathbb{Q})|=8 $, if $ p$  divides $|E_i(\mathbb{F}_{p})|$ , then $8p$ divides $|E_i(\mathbb{F}_{p})|$ and it contradicts the Hasse bound. Since $ p $ splits completely in $ K $, we have $ p \nmid |\tilde{E_{i}}_{,v}(\mathbb{F}_{v})_{p}| $. By Lemma 2.1 and Corollary 3.6, $ f_{E_i}(0) $ is a $ p $-adic unit and $ \lambda_{E_i}=0 $. Hence, $ E_i(K_\infty) $ is finite.\\
     	
     	Case 2.  $ E_i $ has multiplicative reduction at all primes lying above $ p $.\\
     	
     	In this case, we have the formula [\cite{Greenberg1999}, pp. 92-93]:
     	\begin{equation*}
     		f_{E_i}(0) \sim\left(\prod_{v \mid p} l_{v}\right)\left(\prod_{v~\mathrm{bad}} c_{v}^{(p)}\right)\left|\operatorname{Sel}(E_i/K)_{p}\right| /\left|E_i(K)_{p}\right|^{2}
     	\end{equation*}
        If $ E_i $ has non-split multiplicative reduction at $ v $, we can take $ l_v=2 $. If $ E_i $ has split multiplicative reduction at $ v $, we can take
        \begin{equation*}
        l_{v}=\frac{\log _{p}(N_{{K_{v}} / \mathbb{Q}_{p}}(q_{E_i}^{(v)}))}{\operatorname{ord}_{p}(N_{K_{v} / \mathbb{Q}_{p}}(q_{E_i}^{(v)}))} \cdot \frac{[K_{v} \cap \mathbb{Q}_{p}^{\mathrm{unr}}:\mathbb{Q}_{p}]}{2 p[K_{v} \cap \mathbb{Q}_{p}^{\mathrm{cyc}}:\mathbb{Q}_{p}]}	
       \end{equation*}	
       Here $ q_{E_i} $ is the Tate period, and $ \mathbb{Q}_{p}^{\mathrm{unr}} $ and $ \mathbb{Q}_{p}^{\mathrm{cyc}} $ are the maximal unramified extension and the cyclotomic $ \mathbb{Z}_p $-extension of $ \mathbb{Q}_{p} $, respectively. From [\cite{Thorne2019}, proof of Theorem 5], we know that $ \frac{\log _{p}(N_{{K_{v}} / \mathbb{Q}_{p}}(q_{E_i}^{(v)}))}{\operatorname{ord}_{p}(N_{K_{v} / \mathbb{Q}_{p}}(q_{E_i}^{(v)}))} $ lies in $ p\mathbb{Z}_p^{\times} $ since $ p $ splits completely in $ K $. By Lemma 2.1 and Corollary 3.6 and the formula, $ E_i(K_\infty) $ is finite.\\
       
       Case 3. $ E_i $ has good supersingular reduction at all primes lying above $ p $.
       
       In this case, our assertion follows from [\cite{Iovita2006}, Theorem 5.1] as $a_p =0$ under our assumptions.
     \end{proof}
 
     Theorem 1.1 follows from Corollary 3.2 and Proposition 3.7.\\
     
	    At last, we give some data from \cite{LMFDBCollaboration2022}.\\
	    
     \begin{center}
     	
     	\begin{tabular}{|rrrr|}
     		\hline
     		Base field & Mordell-Weil group & Tamagawa product & Analytic order of {\fontencoding{OT2}\selectfont X} \\
     		\hline
     		$ \mathbb{Q}(\sqrt{2}) $ & $ E_i: \mathbb{Z}/{2}\mathbb{Z} \times \mathbb{Z}/{4}\mathbb{Z} $ & $ E_1: 16,~ E_2: 4 $ & $E_i: 1 $ \\
     		$ \mathbb{Q}(\sqrt{6}) $ & $ E_i: \mathbb{Z}/{2}\mathbb{Z} \times \mathbb{Z}/{4}\mathbb{Z} $ & $ E_1: 16,~ E_2: 4 $ & $E_i: 1 $ \\
     		$ \mathbb{Q}(\sqrt{17}) $ & $ E_i: \mathbb{Z}/{2}\mathbb{Z} \times \mathbb{Z}/{4}\mathbb{Z} $ & $ E_1: 16,~ E_2: 4 $ & $E_i: 1 $ \\
     		$ \mathbb{Q}(\sqrt{21}) $ & $ E_i: \mathbb{Z}/{2}\mathbb{Z} \times \mathbb{Z}/{4}\mathbb{Z} $ & $ E_1: 32,~ E_2: 8 $ & $E_i: 1 $ \\
     		$ \mathbb{Q}(\sqrt{53}) $ & $ E_i: \mathbb{Z}/{2}\mathbb{Z} \times \mathbb{Z}/{4}\mathbb{Z} $ & $ E_1: 16,~ E_2: 4 $ & $E_i: 1 $ \\
     		$ \mathbb{Q}(\sqrt{61}) $ & $ E_i: \mathbb{Z}/{2}\mathbb{Z} \times \mathbb{Z}/{4}\mathbb{Z} $ & $ E_1: 64,~ E_2: 16 $ & $E_i: 1 $ \\
     		$ \mathbb{Q}(\sqrt{65}) $ & $ E_i: \mathbb{Z}/{2}\mathbb{Z} \times \mathbb{Z}/{4}\mathbb{Z} $ & $ E_1: 32,~ E_2: 8 $ & $E_i: 1 $ \\
     		$ \mathbb{Q}(\sqrt{69}) $ & $ E_i: \mathbb{Z}/{2}\mathbb{Z} \times \mathbb{Z}/{4}\mathbb{Z} $ & $ E_1: 32,~ E_2: 8 $ & $E_i: 1 $ \\
     		$ \mathbb{Q}(\sqrt{77}) $ & $ E_i: \mathbb{Z}/{2}\mathbb{Z} \times \mathbb{Z}/{4}\mathbb{Z} $ & $ E_1: 16,~ E_2: 4 $ & $E_i: 4 $ \\
     		\hline	
     	\end{tabular}

     \end{center}
    \begin{center}
    	\textbf{Table 2}. $ E_i $ in some real quadratic fields.
    \end{center}
     ~\\
     
	  From table 2, we can find many real quadratic fields and infinitely many odd primes satisfying our assumptions.
	  
	~\\
	
	\begin{ack*}
		The author would like to express appreciation to his supervisor Professor Takeshi Saito for his comments and suggestions. Also, the author would like to thank Professor Jack Thorne and Dr. Sho Yoshikawa for some helpful communications.
	\end{ack*}

	\newpage
	\bibliography{reference.bib}
	\bibliographystyle{alpha}
	~\\
	
	Graduate School of Mathematical Sciences, University of Tokyo, Komaba, Meguro, Tokyo 153-8914, Japan\\
	
	\textit{Email address}: zhangxy96@g.ecc.u-tokyo.ac.jp
	
\end{document}